\documentclass[10pt,leqno]{amsart}
\usepackage{amssymb,amsthm}
\usepackage{amsmath}
\usepackage{setspace}
\usepackage{dcpic,pictexwd}
\usepackage[all]{xy}
\usepackage[pdftex]{graphicx}
\usepackage{mathrsfs}
\usepackage[pdftex]{hyperref}
\usepackage{bbm}

\theoremstyle{definition}
 \newtheorem{definition}{Definition}[section]

\theoremstyle{plain}
 \newtheorem{proposition}[definition]{Proposition}

\theoremstyle{plain}
 \newtheorem{theorem}[definition]{Theorem}
 
\theoremstyle{definition}
 \newtheorem{example}[definition]{Example}

\theoremstyle{plain}
 \newtheorem{lemma}[definition]{Lemma}

\theoremstyle{plain}

\theoremstyle{remark}
 \newtheorem{remark}[definition]{Remark}

\theoremstyle{definition}

\theoremstyle{plain}

\newcommand{\Ext}{\mathrm{Ext}}
\newcommand{\End}{\mathrm{End}}

\newcommand{\Hom}{\mathrm{Hom}}
\newcommand{\SHom}{\underline{\mathrm{Hom}}}

\newcommand{\Ca}{\mathcal{C}}
\newcommand{\Fun}{\mathrm{F}}

\newcommand{\Def}{\mathrm{Def}}
\newcommand{\Sets}{\mathrm{Sets}}

\newcommand{\SEnd}{\underline{\End}}
\newcommand{\A}{\Lambda}
\newcommand{\G}{\Gamma}

\renewcommand{\k}{\Bbbk}
\renewcommand{\1}{\mathbbm{1}}

\hypersetup{
    bookmarks=true,         
    unicode=false,          
    pdftoolbar=true,        
    pdfmenubar=true,        
    pdffitwindow=false,     
    pdfstartview={FitH},    
    pdftitle={Universal Deformation Rings for Complexes Over Finite-Dimensional Algebras},    
    pdfauthor={Jose A. Velez-Marulanda},     
    pdfsubject={Derived Categories, Singularity Categories, Universal deformation rings},   
    pdfcreator={Jose A. Velez-Marulanda},   
    pdfproducer={Jose A. Velez-Marulanda}, 
    pdfkeywords={Derived Categories} {Singularity Categories}{Universal Deformation Rings}, 
    pdfnewwindow=true,      
    colorlinks=true,       
    linkcolor=red,          
    citecolor=blue,        
    filecolor=magenta,      
    urlcolor=blue           
}

\title[Deformations of Gorenstein-projective modules]{On Deformations of Gorenstein-projective modules over Nakayama and triangular matrix algebras} 
\dedicatory{To my friends Wildomar Alarc\'on and Ricardo Pe\~na}
\thanks{The author was supported by the Release Time for Research Scholarship of the Office of Academic Affairs and by the Faculty 
Research Seed Grant funded by the Office of Sponsored Programs \& Research Administration at the Valdosta State University. The author was also partly supported by CODI and Estrategia de Sostenibilidad (Universidad de Antioquia, UdeA) and Colciencias-Ecopetrol (no. 0266-2013)}
\author{Jos\'e A. V\'elez-Marulanda}
\address{Department of Mathematics, Valdosta State University, Valdosta, GA, U.S.A.}
\email{javelezmarulanda@valdosta.edu (Corresponding author)}


\begin{document}
\renewcommand{\labelenumi}{\textup{(\roman{enumi})}}
\renewcommand{\labelenumii}{\textup{(\roman{enumi}.\alph{enumii})}}
\numberwithin{equation}{section}

\begin{abstract}
Let $\k$ be a fixed field of arbitrary characteristic, and let $\A$ be a finite dimensional $\k$-algebra. Assume that $V$ is a left $\A$-module of finite dimension over $\k$. F. M. Bleher and the author previously proved that $V$ has a well-defined versal deformation ring $R(\A,V)$ which is a local complete commutative Noetherian ring with residue field isomorphic to $\k$. Moreover, $R(\A,V)$ is universal if the endomorphism ring of $V$ is isomorphic to $\k$. In this article we prove that if $\A$ is a basic connected Nakayama algebra without simple modules and $V$ is a Gorenstein-projective left $\A$-module, then $R(\A,V)$ is universal. Moreover, we also prove that the universal deformation rings $R(\A,V)$ and $R(\A, \Omega V)$ are isomorphic, where $\Omega V$ denotes the first syzygy of $V$. This result extends the one obtained by F. M. Bleher and D. J. Wackwitz concerning universal deformation rings of finitely generated modules over self-injective Nakayama algebras.  In addition, we also prove the following result concerning versal deformation rings of finitely generated modules over triangular matrix finite dimensional algebras. Let $\Sigma=\begin{pmatrix} \A & B\\0& \G\end{pmatrix}$ be a triangular matrix finite dimensional Gorenstein $\k$-algebra with $\G$ of finite global dimension and $B$ projective as a left $\A$-module. If $\begin{pmatrix} V\\W\end{pmatrix}_f$ is a finitely generated Gorenstein-projective left $\Sigma$-module, then the versal deformation rings $R\left(\Sigma,\begin{pmatrix} V\\W\end{pmatrix}_f\right)$ and $R(\A,V)$ are isomorphic. 
\end{abstract}
\maketitle

\section{Introduction}\label{int}
Throughout this article, we assume that $\k$ is a fixed field of arbitrary characteristic, and that all our modules are finitely generated. We denote by $\hat{\Ca}$ the category of all complete local commutative Noetherian $\k$-algebras with residue field $\k$. In particular, the morphisms in $\hat{\Ca}$ are continuous $\k$-algebra homomorphisms that induce the identity map on $\k$.  Let $\A$ be a fixed finite dimensional $\k$-algebra. For all objects $R$ in $\hat{\Ca}$, we denote by $R\A$ the tensor product of $\k$-algebras $R\otimes_\k\A$. Note that $R\A$ is an $R$-algebra, and if $R$ is an Artinian ring then $R\A$ is also Artinian (both on the left and the right sides). We denote by $R\A$-mod the category of finitely generated left $R\A$-modules.  If $\mathcal{M}$ is a class of left  $R\A$-modules, we denote by $\mathrm{add}\, \mathcal{M}$ the smallest subcategory of $R\A$-mod containing $\mathcal{M}$. In particular, for all left $R\A$-modules $X$, $\mathrm{add}\, X$ denotes the full subcategory of $R\A$-mod whose objects are direct summands of finite direct sums of copies of $X$.  We denote by $\A$-\underline{mod} the stable category of $\A$, i.e., the objects of $\A$-\underline{mod} are the same as those of $\A$-mod, and two morphisms $f,g:X\to Y$ in $\A$-\underline{mod} are identified provided that $f-g$ factors through a projective left $\A$-module. Let $V$ be a fixed left $\A$-module. We denote by $|V|$ the length of $V$ as a $\A$-module, and by $\End_\A(V)$ (resp. by $\SEnd_\A(V)$) the 
endomorphism ring (resp. the stable endomorphism ring) of $V$. We denote by $\Omega V$ or by $\Omega_\A V$ the first syzygy of $V$ as a left $\A$-module, i.e., $\Omega V$ is the kernel of a projective cover $P(V)\to V$ of $V$ over $\A$, which is unique up to isomorphism. Let $R$ be an arbitrary object in $\hat{\Ca}$. A {\it lift} of $V$ over $R$ is a left $R\A$-module $M$ that is free over $R$ together with a $\A$-module isomorphism $\phi: \k\otimes_RM\to V$.   A {\it deformation} of $V$ over $R$ is defined to be an isomorphism class of lifts of $V$ over $R$. It follows from \cite[Prop. 2.1 \& Prop. 2.5]{blehervelez} that if $V$ is a left $\A$-module, then $V$ has a well-defined versal deformation ring $R(\A,V)$, which is an object in $\hat{\Ca}$, and which also an invariant under Morita equivalence. More recently, F. M. Bleher and D. J. Wackwitz proved in \cite[Prop. 2.4]{bleher15} that if $\A$ is a Frobenius $\k$-algebra (i.e., ${_\A}\A$ and $\Hom_\k(\A{_\A},\k)$ are isomorphic as left $\A$-modules) and $V$ is non-projective, then the versal deformation rings $R(\A,V)$ and $R(\A,\Omega V)$ are isomorphic in $\hat{\Ca}$, which generalizes \cite[Thm. 2.6 (iv)]{blehervelez}. Moreover, they also proved in \cite[Thm. 1.3]{bleher15} that if $\A$ is an indecomposable $\k$-algebra that is stably Morita equivalent (in the sense of \cite{broue}) to a self-injective split basic Nakayama algebra and $V$ is further indecomposable, then $R(\A,V)$ is universal and its isomorphism class is either $\k$ or $\k[\![t]\!]/(t^2)$, or a quotient of a ring of power series in finitely many variables and which is also determined by the closets distance of the isomorphism class of $V$ to the boundary of the stable Auslander-Reiten quiver of $\A$.

Following \cite{enochs0,enochs}, we say that $V$ is {\it Gorenstein-projective} provided that there exists an acyclic complex of projective left $\A$-modules 
\begin{equation}\label{completeres}
P^\bullet: \cdots\to P^{-2}\xrightarrow{f^{-2}} P^{-1}\xrightarrow{f^{-1}} P^0\xrightarrow{f^0}P^1\xrightarrow{f^1}P^2\to\cdots
\end{equation}  
such that $\Hom_\A(P^\bullet, \A)$ is also acyclic and $V=\mathrm{coker}\,f^0$. In this situation, we say that $P^\bullet$ is a {\it complete $\A$-projective resolution}.  
We denote by $\A$-Gproj (resp. by $\A$-\underline{Gproj}) the full subcategory of $\A$-mod (resp. of $\A$-\underline{mod}) consisting of finitely generated Gorenstein-projective left $\A$-modules. It is well-known that $\A$-Gproj is a Frobenius category in the sense of \cite[Chap. I, \S 2]{happel}, and that consequently, $\A$-\underline{Gproj} is a triangulated category. 
Recall that $\A$ is said to be a {\it Gorenstein} $\k$-algebra provided that $\A$ has finite injective dimension as a left and right $\A$-module. In particular, algebras of finite global dimension as well as self-injective algebras are Gorenstein.  It follows from \cite[Thm. 1.2]{bekkert-giraldo-velez} that versal deformation rings of Gorenstein-projective modules are preserved under singular equivalence of Morita type between Gorenstein $\k$-algebras (cf.  \cite[Prop. 3.2.6]{blehervelez2}). These singular equivalences of Morita type were introduced by X. W. Chen and L. G. Sun in \cite{chensun}, and then further discussed by G. Zhou and A. Zimmermann in \cite{zhouzimm}, as a way of generalizing the concept of stable equivalence of Morita type. It was also proved in \cite[Thm. 1.2]{bekkert-giraldo-velez} that if $V$ is a Gorenstein-projective left $\A$-module with $\SEnd_\A(V)=\k$, then $R(\A,V)$ is universal, which generalizes \cite[Thm. 2.6]{blehervelez}.  On the other hand, it follows from \cite[Thm. 5.2]{bekkert-giraldo-velez} that if $\A$ is a monomial algebra in which there is no overlap (in the sense of \cite{chen-shen-zhou}) and $V$ is a  non-projective Gorenstein-projective $\A$-module, then  $R(\A,V)$ is universal and isomorphic either to $\k$ when $V\not=\Omega V$ or to $\k[\![t]\!]/(t^2)$  otherwise.  All of these facts raise the question of for which other finite dimensional $\k$-algebras every indecomposable non-projective Gorenstein-projective $\A$-module $V$, the versal deformation ring $R(\A, V)$ is universal.

Recall that $\A$ is a {\it Nakayama} algebra provided that all indecomposable projective and all indecomposable injective $\A$-modules are uniserial, or equivalently, all indecomposable $\A$-modules are uniserial.   It follows from e.g. \cite[Thm.V.3.2]{assem3}, that $\A$ is a basic and connected Nakayama $\k$-algebra with exactly $s(\A)=s\geq 1$ isomorphism classes of indecomposable simple left $\A$-modules and without simple projective objects if and only if $\A$ is of the form $\k Q/I$, where $Q$ is the quiver
\begin{align}\label{cyclequiver}
Q:& \xymatrix@1@=40pt{ \underset{0}{\bullet} \ar[r]^{\alpha_0} & \underset{1}{\bullet} \ar[r]^{\alpha_1} &\underset{2}{\bullet}\ar@{.}[r]&\underset{s-2}{\bullet}\ar[r]^{\alpha_{s-2}} &  \underset{s-1}{\bullet}  \ar@/^2pc/[llll]^{\alpha_{s-1}} }, 
\end{align}
and $I$ is an admissible ideal of $\k Q$.

Let $\A$ be a basic connected Nakayama $\k$-algebra without simple projective modules. In \cite{ringel}, C. M. Ringel provided a description of the Gorenstein-projective modules over $\A$. Moreover, he also proved the following result (see \cite[Prop. 1]{ringel}). Let $\mathscr{C}(\A)$ be the category of left $\A$-modules such that the indecomposable objects are all the non-projective Gorenstein-projective indecomposable left $\A$-modules as well as their projective covers. Then $\mathscr{C}(\A)$ is a full exact abelian subcategory of $\A$-mod which is closed under extensions and projective covers and which is also equivalent to the abelian category $\A'$-mod, where $\A'$ is a basic connected self-injective Nakayama $\k$-algebra. Ringel calls $\mathscr{C}(\A)$ the {\it Gorenstein core} of $\A$.  Let $\mathscr{E}(\A)$ be the class of non-zero indecomposable Gorenstein-projective left $\A$-modules $E$ such that no proper non-zero factor module of $E$ is a Gorenstein-projective left $\A$-module. Then the objects in $\mathscr{E}(\A)$ are the simple objects in $\mathscr{C}(\A)$, which Ringel call the {\it elementary} Gorenstein-projective modules of $\A$. It follows from \cite[Prop. 2 (a)]{ringel} that every non-zero object $V$ in $\mathscr{C}(\A)$ has a filtration with composition factors in $\mathscr{E}(\A)$ and thus $\mathscr{C}(\A)$ is an abelian length category in the sense of \cite{gabriel3}. Note that since $\mathscr{C}(\A)$ is also an exact Krull-Schmidt category, we can talk about the Auslander-Reiten quiver of $\mathscr{C}(\A)$ (see \cite[\S 2.3]{ringel2}). Let $E_1,\ldots,E_g$ be a complete list of isomorphism-class representatives in $\mathscr{E}(\A)$, and for all $1\leq i\leq g$, let $P(E_i)$ be the projective $\A$-module cover of $E_i$. Assume that $s(\A)<|P(E_i)|$ for all $1\leq i\leq g$. By \cite[Prop. 2 (b)]{ringel}, $s(\A)=\sum_{i=1}^g|E_i|$, and by the arguments in \cite[pg. 252]{ringel}, $s(\A)$ divides $\sum_{i=1}^g|P(E_i)|$. Let $V$ be a non-zero indecomposable object in $\mathscr{C}(\A)$. Since $V$ has a unique composition series as a $\A$-module, then it has a unique composition series inside the category $\mathscr{C}(\A)$. 
In order to give an explicit description of the versal deformation rings of the indecomposable non-projective Gorenstein-projective left $\A$-modules, we need the following definition from \cite[Def. 1.1 (b)]{bleher5}.
\begin{definition}\label{defi0}
For all integers $n\geq 1$, we denote by $N_n$ the $n\times n$ matrix with entries in the power series algebra $\k[\![t_1,\ldots,t_n]\!]$ defined by 
\begin{equation*}
N_n=\begin{pmatrix}0&\cdots&0&t_n\\&&&t_{n-1}\\&I_{n-1}&&\vdots\\&&&t_1\end{pmatrix}
\end{equation*}
where $I_{n-1}$ is the $(n-1)\times (n-1)$ identity matrix. In particular, $N_1=(t_1)$. For all integers $a\geq 0$ and $n\geq 1$, define $J_n(a)$ to be the ideal of $\k[\![t_1,\ldots,t_n]\!]$ generated by the entries of $(N_n)^a$. If $n=0$, then we let $J_0(a)$ to be the zero-ideal of $\k$. 
\end{definition}

Assume that $\mathscr{C}(\A)\not=0$. The first goal of this note is to prove the following result which shows that every versal deformation ring $R(\A,V)$ of a non-projective indecomposable object $V$ in $\mathscr{C}(\A)$ is universal and stable under syzygies. Moreover, it also provides an explicit description of $R(\A,V)$ as a quotient algebra of a ring of power series over $\k$. 

\begin{theorem}\label{thm1}
Let $\A$ be a basic connected Nakayama $\k$-algebra without simple projective modules and with $\mathscr{C}(\A)\not=0$. Let $E_1,\ldots, E_g$ be a complete list of isomorphism-class representatives in $\mathscr{E}(\A)$. For all $1\leq i\leq g$, let $|P(E_i)|$ be the length of the projective $\A$-module cover of $E_i$,  and assume that $s(\A)<|P(E_i)|$.
 \begin{enumerate}
\item There exists a basic connected self-injective Nakayama $\k$-algebra $\A'$ and a left $\A'$-module $V'$ such that the versal deformation ring $R(\A,V)$ is isomorphic to $R(\A', V')$ in $\hat{\Ca}$. In particular, $R(\A,V)$ is also universal and isomorphic to $R(\A,\Omega V)$ in $\hat{\Ca}$.
\item Let $\ell_{\mathscr{C}(\A)}$ be the integer
\begin{equation*}
\ell_{\mathscr{C}(\A)}=\frac{1}{s(\A)}\sum_{i=1}^g|P(E_i)|.
\end{equation*}
Let $\mu, \ell' \geq 0$ be integers such that $\ell'\leq g-1$ and 
\begin{equation*}
\ell_{\mathscr{C}(\A)}=\mu g+\ell'.
\end{equation*}
Let $V$ be an indecomposable non-projective object in $\mathscr{C}(\A)$ and denote by $d_{\mathscr{C}(\A),V}$ the distance of the isomorphism class of $V$ to the closest boundary of the stable Auslander-Reiten quiver of $\mathscr{C}(\A)$, and let $\ell_{\mathscr{C}(\A),V}=d_{\mathscr{C}(\A),V}+1$. Write $\ell_{\mathscr{C}(\A),V}$ as 
\begin{equation*}
\ell_{\mathscr{C}(\A),V}=ng+i,
\end{equation*}
where $n,i\geq 0$ and $i\leq g-1$. Then the universal deformation ring $R(\A,V)$ is isomorphic in $\hat{\Ca}$ to $\k[\![t_1,\ldots,t_n]\!]/J_n(m_V)$, where 
\begin{equation*}
m_V=\begin{cases} \mu, &\text{ if $0\leq i\leq \ell'$,}\\
\mu-1, &\text{otherwise,}\end{cases}
\end{equation*}
where $J_n(m_V)$ is as in Definition \ref{defi0}. 
\end{enumerate} 
\end{theorem}

Observe that Theorem \ref{thm1} extends the results of \cite[Thm. 1.3]{bleher15} to indecomposable non-projective Gorenstein-projective modules over arbitrary basic connected Nakayama algebras without simple projective modules.

Recall that  $\Sigma$ is a finite dimensional triangular matrix $\k$-algebra if $\Sigma$ is of the form 
\begin{equation}\label{triangalgebra}
\Sigma=\begin{pmatrix} \A & B\\0 & \G\end{pmatrix}, 
\end{equation}
where $\A$ and $\G$ are finite dimensional $\k$-algebras and $B$ is a $\A$-$\G$-bimodule. Recall also that the product in $\Sigma$ is induced by the usual product of $2\times 2$ square matrices. For a brief description of the module category of finitely generated modules over triangular matrix algebras, see \S \ref{ss3}.  In \cite[Thm. 1.4]{zhang}, P. Zhang provided the following explicit description of the Gorenstein-projective left $\Sigma$-modules for when $B$ is a compatible $\A$-$\G$-bimodule (see \cite[Def. 1.1]{zhang}). A left $\Sigma$-module $\begin{pmatrix} V\\W\end{pmatrix}_f$ is Gorenstein-projective if and only if $f:B\otimes_\G W\to V$ is injective, $W$ is a Gorenstein-projective left $\G$-module, and $\mathrm{coker}\, f$ is also a Gorenstein-projective left $\A$-module. In this situation $V$ is a Gorenstein-projective left $\A$-module if and only if $B\otimes_\G W$ is also a Gorenstein-projective left $\A$-module. Consequently, if $\Sigma$ is Gorenstein and $\G$ has finite global dimension, then the triangulated categories $\Sigma$-\underline{Gproj} and $\A$-\underline{Gproj} are equivalent (see \cite[Cor. 2.7 (i)]{zhang}).  
On the other hand, it follows by \cite[Thm. 2.2 (i)]{xiong-zhang} that if $B$ has finite projective dimension as a left $\A$-module, then $\Sigma$ is Gorenstein if and only if $\A$ and $\G$ are both Gorenstein and $B$ has finite projective dimension as a right $\G$-module.

The second goal of this note is to apply \cite[Cor. 2.7 (i)]{zhang} to prove the following result concerning versal deformation rings of Gorenstein-projective modules over finite dimensional Gorenstein triangular matrix $\k$-algebras.  

\begin{theorem}\label{thm2}
Let $\Sigma$ be a Gorenstein triangular matrix $\k$-algebra as in (\ref{triangalgebra}) with $B$ projective as a left $\A$-module and with $\G$ of finite global dimension. Let $\begin{pmatrix} V\\W\end{pmatrix}_f$ be a Gorenstein-projective left $\Sigma$-module. Then the versal deformation rings $R\left(\Sigma, \begin{pmatrix} V\\W\end{pmatrix}_f\right)$ and $R(\A,V)$ are isomorphic in $\hat{\Ca}$.
\end{theorem}



This note is meant to be a continuation of \cite{bekkert-giraldo-velez} and is organized as follows. In \S\ref{sec2}, we review the preliminary results concerning lifts and deformations of finitely generated modules in the sense of \cite{blehervelez} as well as the description of finitely generated Gorenstein-projective modules over Nakayama and triangular matrix algebras as given in \cite{ringel} and \cite{zhang}, respectively. In \S\ref{sec3}, we give the proofs of Theorems \ref{thm1} and \ref{thm2}. We refer the reader to look at e.g. \cite{chen-shen-zhou, holmH} (and their references) for basic concepts concerning Gorenstein-projective modules. For basic concepts from the representation theory of algebras such as projective covers, syzygies of modules, stable categories and homological dimension of modules over finite dimensional algebras, we refer the reader to \cite{assem3,auslander,curtis,weibel}.

\section{Preliminaries}\label{sec2}
Let $\k$ and $\hat{\Ca}$ be as in \S \ref{int}. Let $\A$ be a finite dimensional $\k$-algebra, and let $V$ be a left $\A$-module. 

\subsection{Lifts, deformations, and (uni)versal deformation rings}\label{sec21}

Let $R$ be a ring in $\hat{\Ca}$. Following \cite{blehervelez}, a {\it lift} $(M,\phi)$ 
of $V$ over $R$ is a left $R\A$-module $M$ 
that is free over $R$ 
together with an isomorphism of $\A$-modules $\phi:\k\otimes_RM\to V$. Two lifts $(M,\phi)$ and $(M',\phi')$ over $R$ are {\it isomorphic} 
if there exists an $R\A$-module 
isomorphism $f:M\to M'$ such that $\phi'\circ (\mathrm{id}_\k\otimes_R f)=\phi$.
If $(M,\phi)$ is a lift of $V$ over $R$, we  denote by $[M,\phi]$ its isomorphism class and say that $[M,\phi]$ is a {\it deformation} of $V$ 
over $R$. We denote by $\Def_\A(V,R)$ the 
set of all deformations of $V$ over $R$. The {\it deformation functor} over $V$ is the 
covariant functor $\hat{\Fun}_V:\hat{\Ca}\to \Sets$ defined as follows: for all rings $R$ in $\hat{\Ca}$, define $\hat{\Fun}_V(R)=\Def_
\A(V,R)$, and for all morphisms $\theta:R\to 
R'$ in $\hat{\Ca}$, 
let $\hat{\Fun}_V(\theta):\Def_\A(V,R)\to \Def_\A(V,R')$ be defined as $\hat{\Fun}_V(\theta)([M,\phi])=[R'\otimes_{R,\theta}M,\phi_\theta]$, 
where $\phi_\theta: \k\otimes_{R'}
(R'\otimes_{R,\theta}M)\to V$ is the composition of $\A$-module isomorphisms 
\[\k\otimes_{R'}(R'\otimes_{R,\theta}M)\cong \k\otimes_RM\xrightarrow{\phi} V.\]  

Following \cite[\S 2.6]{sch}, we call the set $\hat{\Fun}_V(\k[\![t]\!]/(t^2))$ the tangent space of $\hat{\Fun}_V$, which has an structure of a $\k$-vector space by \cite[Lemma 2.10]{sch}. 

Suppose there exists a ring $R(\A,V)$ in $\hat{\Ca}$ and a deformation $[U(\A,V), \phi_{U(\A,V)}]$ of $V$ over $R(\A,V)$ with the 
following property. For each $R$ in $\hat{\Ca}$ and for all deformations  $[M,\phi]$ of $V$ over $R$, there exists a morphism $\psi_{R(\A,V), R, [M,\phi]}:R(\A,V)\to R$ 
in $\hat{\Ca}$ such that 
\[\hat{\Fun}_V(\psi_{R(\A,V), R, [M,\phi]})[U(\A,V), \phi_{U(\A,V)}]=[M,\phi],\]
and moreover, $\psi_{R(\A,V), R,[M,\phi]}$ is unique if $R$ is the ring of dual numbers $\k[\epsilon]$ with $\epsilon^2=0$.  Then $R(\A,V)$ and $
[U(\A,V),\phi_{U(\A,V)}]$ are called  the {\it versal deformation ring} and {\it versal deformation} of $V$, respectively. If the morphism $
\psi_{R(\A,V),R,[M,\phi]}$ is unique for all rings $R$ in $\hat{\Ca}$ and deformations  $[M,\phi]$ of $V$ over $R$, then $R(\A,V)$ and $[U(\A,V),\phi_{U(\A,V)}]$ are 
called the {\it universal deformation ring} and the {\it universal deformation} of $V$, respectively.  In other words, the universal deformation 
ring $R(\A,V)$ represents the deformation functor $\hat{\Fun}_V$ in the sense that $\hat{\Fun}_V$ is naturally isomorphic to the $\Hom$ 
functor $\Hom_{\hat{\Ca}}(R(\A,V),-)$. Using Schlessinger's criteria \cite[Thm. 2.11]{sch} and using methods similar to those in \cite{mazur}, 
it is straightforward to prove that the deformation functor $\hat{\Fun}_V$ is continuous (see \cite[\S 14]{mazur} for the definition), that there exists an isomorphism of $\k$-vector spaces 
\begin{equation}\label{hoch}
\hat{\Fun}_V(\k[\![t]\!]/(t^2))\to \Ext_\A^1(V,V),
\end{equation}
that $V$ has a versal deformation ring which is an invariant under Morita equivalence (see \cite[Prop. 2.5]
{blehervelez}), and that this versal deformation ring is universal provided that the 
endomorphism ring $\End_\A(V)$ is isomorphic to $\k$ (see \cite[Prop. 2.1]{blehervelez}). 

\begin{remark}\label{univ}
\begin{enumerate}
\item If follows from the isomorphism of $\k$-vector spaces (\ref{hoch}) (see also  \cite[Remark 2.1]{bleher15}) that if $\Ext_\A^1(V,V)\cong \k$, then the versal deformation ring $R(\A,V)$ is universal and isomorphic to $\k$. In particular, if $P$ is a projective left $\A$-module, then $R(\A,P)$ is universal and isomorphic to $\k$. It also follows from (\ref{hoch}) that if $\Ext_\A^1(V,V)\cong \k$, then the versal deformation $R(\A,V)$ is isomorphic to a quotient $\k$-algebra of $\k[\![t]\!]$. 
\item It was proved in \cite[Thm. 1.2 (ii)]{bekkert-giraldo-velez} that if $V$ is a Gorenstein-projective left $\A$-module whose stable endomorphism ring $\SEnd_\A(V)$ is isomorphic to $\k$, then $R(\A,V)$ is universal. This generalizes \cite[Thm. 2.6 (ii)]{blehervelez}.
\item If $\A$ is a Frobenius $\k$-algebra and $V$ is non-projective, then it follows from \cite[Prop. 2.4]{bleher15} that the versal deformation rings $R(\A,V)$ and $R(\A,\Omega V)$ are isomorphic in $\hat{\Ca}$. This generalizes \cite[Thm. 2.6 (iv)]{blehervelez}. 
\end{enumerate}
\end{remark}

\begin{remark}\label{ProjR}
Let $R$ be an Artinian ring in $\hat{\Ca}$, let $\iota_R:\k\to R$ be the unique morphism in $\hat{\Ca}$ endowing $R$ with a $\k$-algebra structure, and let $\pi_R:R\to \k$ be the natural projection in $\hat{\Ca}$. Then $\pi_R\circ \iota_R =\mathrm{id}_\k$. 
\begin{enumerate}
\item For all projective left $\A$-modules $P$, we let $P_R=R\otimes_{\k,\iota_R}P$. Then $P_R$ is a projective left $R\A$-module cover of $P$, and $(P_R, \pi_{R,P})$ is a lift of $P$ over $R$, where $\pi_{R,P}$ is the natural isomorphism $\k\otimes_{R, \pi_R}(R\otimes_{\k,\iota_R}P)\to P$.
\item Let $\epsilon: P(V)\to V$ be a projective left $\A$-module cover of $V$ (which is unique up to isomorphism) and let $(M, \phi)$ be a lift of $V$ over $R$. By (i), we have that $P_R(V)=R\otimes_{\k, \iota_R}P(V)$ is a projective left  $R\A$-module cover of $P(V)$. Since $\epsilon$ is an essential epimorphism, there exists an epimorphism of $R\A$-modules $\alpha_R: P_R(V)\to M$ such that $\phi\circ (\mathrm{id}_\k\otimes\alpha_R)= \epsilon\circ \pi_{R,P(V)}$. Let  $\Omega_RM=\ker \alpha$. Note that since $M$ and $P_R(V)$ are both free over $R$, then $\Omega_R(M)$ is also free over $R$. Moreover,  there exists an isomorphism of left $\A$-modules $\Omega_R(\phi):\k\otimes_R\Omega_R(M)\to \Omega V$ such that $\pi_{R,P(V)}\circ \mathrm{id}_\k\otimes\beta_R=\iota_V\circ \Omega_R(\phi)$, where $\iota_V:\Omega V\to P(V)$ and $\beta_R:\Omega_R(M)\to P_R(V)$ are the natural inclusions. In particular $(\Omega_R(M), \Omega_R(\phi))$ is a lift of $\Omega V$ over $R$. 
\end{enumerate}  
\end{remark}

\subsection{Gorenstein-projective modules over Nakayama algebras}\label{nakayama}
 

Let $\A$, $\mathscr{C}(\A)$ and $\mathscr{E}(\A)=\{E_1,\ldots,E_g\}$  be as in the hypothesis of Theorem \ref{thm1}, and let $\mathscr{X}(\A)$ be the class of simple left $\A$-modules $S$ such that the projective $\A$-module cover $P(S)$ of $S$ belongs to $\mathscr{C}(\A)$. Let $\tau_\A$ and $\Gamma(\A)$ denote the Auslander-Reiten translation and the Auslander-Reiten quiver of $\A$, respectively, and let $\tau_\A^-\mathscr{X}(\A)=\{\tau_\A^-S\,|\, S\in \mathscr{X}(\A)\}$. By \cite[Prop. 3]{ringel}, $V\not=0$ is in $\mathscr{C}(\A)$ if and only if $\mathrm{top}\, V$ is in $\mathrm{add}\,\mathscr{X}(\A)$ and $\mathrm{soc}\, V$ is in $\mathrm{add}\,\tau_\A^-\mathscr{X}(\A)$ if and only if $\mathrm{top}\,V$ and $\mathrm{top}\,\Omega V$ are both in $\mathrm{add}\, \mathscr{X}(\A)$. 

\begin{remark}\label{rem0}
By using \cite[Prop. 3]{ringel} mentioned above, the Auslander-Reiten quiver of $\mathscr{C}(\A)$ can be obtained from $\Gamma(\A)$ by deleting the rays in $\Gamma(\A)$ consisting of the indecomposable $\A$-modules with top not in $\mathscr{X}(\A)$, as well as the corays consisting of the indecomposable $\A$-modules with socle not in $\tau_\A^-\mathscr{X}(\A)$. 
\end{remark}

By \cite[Prop. 4]{ringel}, a non-zero left $\A$-module $V$ is Gorenstein-projective if and only if there exists an exact sequence of left $\A$-modules
\begin{equation*}
0\to V\to P_{n-1}\to \cdots \to P_0\to V\to 0, 
\end{equation*} 
where each $P_i$ is a {\it minimal projective} left $\A$-module, i.e., no proper non-zero submodule of $P_i$ is projective.  

Let
\begin{align}\label{projective-generator}
\begin{cases}
\begin{matrix}
P_{\mathscr{X}(\A)}=\bigoplus_{S\in \mathscr{X}(\A)}P(S),&\\\\\A'=\End_\A(P_{\mathscr{X}(\A)})^\text{op},& \text{and}\\\\
H_{\mathscr{X}(\A)}(-)=\Hom_\A(P_{\mathscr{X}(\A)},-): \A\textup{-mod}\to \A'\textup{-mod}.&
\end{matrix}
\end{cases}
\end{align}

\begin{remark}\label{rem00}
By \cite[\S II.1]{auslander}, the left $\A$-module $P_{\mathscr{X}(\A)}$ is also a right $\A'$-module. Moreover, $P_{\mathscr{X}(\A)}$ is a projective generator for $\mathscr{C}(\A)$, i.e., for all objects $V$ in $\mathscr{C}(\A)$ there exists an epimorphism of $\A$-modules $P'\to V$, where $P'$ is an object in $\mathrm{add}\, P_{\mathscr{X}(\A)}$, and thus $H_{\mathscr{X}(\A)}(-)$ restricted to $\mathscr{C}(\A)$ induces an equivalence of categories $\mathscr{C}(\A)\cong \A'\textup{-mod}$ whose quasi-inverse is given by $P_{\mathscr{X}(\A)}\otimes_{\A'}-$. It follows from \cite[Prop. 1]{ringel} that $\A'$ is also a basic connected self-injective Nakayama $\k$-algebra  (cf. \cite[Prop. 3.8]{shen}). Moreover, by \cite[p. 252]{ringel} we obtain that 
$\A'$ has exactly $e=e(\A')=g$ isomorphism classes of simple left $\A'$-modules and the length $\ell\ell(\A')$ of the indecomposable projective left $\A'$-modules (i.e., the Loewy length of $\A'$) is given by 
\begin{equation}\label{lengthself}
\ell\ell(\A')=\frac{1}{s}\sum_{i=1}^g|P(E_i)|, 
\end{equation} 
where as before, $s=s(\A)$ is the number of isomorphism classes of simple left $\A$-modules (see \cite[p. 252]{ringel} for more details). Therefore if $\ell_{\mathscr{C}(\A)}$ is as in Theorem \ref{thm1} (ii), then $\ell_{\mathscr{C}(\A)}=\ell\ell(\A')$. Moreover, if $V$ is an indecomposable non-projective object in $\mathscr{C}(\A)$, and $d_{\mathscr{C}(\A),V}$ is as in Theorem \ref{thm1} (ii), then $d_{\mathscr{C}(\A),V}=d_{V'}$, where $V'$ is the indecomposable non-projective left $\A'$-module $H_{\mathscr{X}(\A)}(V)$ and $d_{V'}$ denotes the closest distance of the isomorphism class of $V'$ to the boundary of the stable Auslander-Reiten quiver of $\A'$.  
\end{remark}
\subsection{Gorenstein-projective modules over triangular matrix algebras}\label{ss3}

Let $\Sigma$ be a triangular matrix $\k$-algebra as in (\ref{triangalgebra}), and let $R$ be a ring in $\hat{\Ca}$ be fixed but arbitrary. 
The morphism
\begin{align*}
\Phi: R\Sigma\to \begin{pmatrix}R\A & B_R\\ 0& R\G \end{pmatrix}
\end{align*}
defined as $\Phi(r\otimes \begin{pmatrix} \lambda & b\\0& \gamma\end{pmatrix})=\begin{pmatrix} r\otimes \lambda & r\otimes b\\0& r\otimes \gamma\end{pmatrix}$ for all $r\in R$, $\lambda\in \A$, $b\in B$ and $\gamma \in \G$ is an isomorphism of $R$-algebras, and thus 
\begin{equation}
R\Sigma\cong \begin{pmatrix} R\A & B_R\\0 & R\G\end{pmatrix}, 
\end{equation} 
where $B_R$ is the $R\A$-$R\G$-bimodule $R\otimes_\k B$. 

In the following, we recall the description of the objects and morphisms in $R\Sigma$-mod (for more details see \cite[Chap. III, \S2]{auslander}). A left $R\Sigma$-module is of the form $\begin{pmatrix} M\\ N\end{pmatrix}_g$, where $M$ is a left $R\A$-module, $N$ is a left $R\G$-module and $g: B_R\otimes_{R\G} N\to M$ is a $R\A$-module homomorphism.  An $R\Sigma$-module homomorphism between two left $R\Sigma$-modules $\begin{pmatrix} M\\N\end{pmatrix}_g$ and $\begin{pmatrix} M'\\N'\end{pmatrix}_{g'}$ is of the form  $\begin{pmatrix} \alpha\\\beta\end{pmatrix}: \begin{pmatrix} M\\N\end{pmatrix}_g\to \begin{pmatrix} M'\\N'\end{pmatrix}_{g'}$, where $\alpha: M\to M'$ is a $R\A$-module homomorphism, $\beta:N\to N'$ is a $R\G$-module homomorphism, and $g'\circ (\mathrm{id}_{B_R}\otimes \beta )=\alpha\circ f$.
A sequence $0\to\begin{pmatrix} M\\N\end{pmatrix}_g\xrightarrow{\tiny\begin{pmatrix} \alpha\\\beta\end{pmatrix}}\begin{pmatrix} M'\\N'\end{pmatrix}_{g'}\xrightarrow{\tiny\begin{pmatrix} \alpha'\\\beta'\end{pmatrix}}\begin{pmatrix} M''\\N''\end{pmatrix}_{g''}\to 0$ of left $R\Sigma$-modules is exact if and only if  $0\to M\xrightarrow{\alpha} M'\xrightarrow{\alpha'}M''\to 0$ and $0\to N\xrightarrow{\beta} N'\xrightarrow{\beta'} N''\to 0$ are exact sequences of left $R\A$-modules and left $R\G$-modules, respectively. The isomorphism classes of the indecomposable projective left $R\Sigma$-modules are induced by the left $R\Sigma$-modules of the form $\begin{pmatrix} \widetilde{P}\\0 \end{pmatrix}_0$ and $\begin{pmatrix}B_R\otimes_{R\G} \widetilde{Q}\\ \widetilde{Q}\end{pmatrix}_{\mathrm{id}_{\widetilde{Q}}}$, where $\widetilde{P}$ is an indecomposable projective left $R\A$-module, and $\widetilde{Q}$ is an indecomposable projective left $R\G$-module. 

Following \cite[Def. 1.1]{zhang}, we say that the $\A$-$\G$-bimodule $B$ is {\it compatible} if satisfies the following conditions: 
\begin{enumerate}
\item If $Q^\bullet$ is a exact sequence of projective $\G$-modules, then $B\otimes_\G Q^\bullet$ is also exact.
\item If $P^\bullet$ is a complete $\A$-projective resolution as in (\ref{completeres}), then $\Hom_\A(P^\bullet, B)$ is also exact. 
\end{enumerate}
In particular, if $B$ has finite projective dimension as a left $\A$-module and as a right $\G$-module, then $B$ is compatible (see \cite[Prop. 1.3]{zhang}). The following description of Gorenstein-projective modules over triangular matrix $\k$-algebras is due to P. Zhang (see \cite[Thm. 1.4]{zhang}).

\begin{theorem}\label{zhangtheorem}
Let $\Sigma$ be a triangular matrix $\k$ algebra as in (\ref{triangalgebra}) with $B$ a compatible $\A$-$\G$-bimodule, and let $\begin{pmatrix} V\\W\end{pmatrix}_f$ be a left $\Sigma$-module. Then $\begin{pmatrix} V\\W\end{pmatrix}_f$ is Gorenstein-projective if and only if the $\A$-module homomorphism $f: B\otimes_\G W\to V$ is injective, $\mathrm{coker}\,f$ is a Gorenstein-projective left $\A$-module, and $W$ is a Gorenstein-projective left $\G$-module. Moreover, $V$ is a Gorenstein-projective left $\A$-module if and only if $B\otimes_\G W$ is also a Gorenstein-projective left $\A$-module.    
\end{theorem}

Recall that by \cite[Thm. 2.2 (i)]{xiong-zhang} we have that if $B$ has finite projective dimension as a left $\A$-module, then $\Sigma$ is Gorenstein if and only if $\A$ and $\G$ are both Gorenstein, and $B$ has also finite projective dimension as a right $\G$-module. It is conjectured in \cite[Conjecture I]{zhang} that $\Sigma$ is Gorenstein if and only if $\A$ and $\G$ are both Gorenstein and $B$ has finite projective dimension as a left $\A$-module and as a right $\G$-module. This is closely related to the {\it Gorenstein Symmetric Conjecture}, which claims that for an arbitrary Artinian algebra $A$, the injective dimension of $A$ as a left $A$-module is finite if and only if the injective dimension of $A$ as a right $A$-module is also finite (see \cite[\S 2]{zhang} for more details).

\begin{remark}
Assume that $\Sigma$ is Gorenstein with $B$ projective as a left $\A$-module and $\G$ of finite global dimension. Let $\begin{pmatrix} V\\W\end{pmatrix}_f$ be a Gorenstein-projective left $\Sigma$-module. Then by Theorem \ref{zhangtheorem}, it follows that $W$ is a Gorenstein-projective and thus a projective left $\G$-module. This implies that $B\otimes_\G W$ is a projective left $\A$-module. It follows again by Theorem \ref{zhangtheorem} that $V$ is then a Gorenstein-projective left $\A$-module. Then by \cite[Cor. 2.7 (i)]{zhang}, we obtain that the operator
\begin{equation*}\label{functorgor}
i^!: \Sigma\textup{-Gproj}\to \A\textup{-Gproj}
\end{equation*}
defined by $\begin{pmatrix} V\\W\end{pmatrix}_f\mapsto V$ induces an equivalence of stable categories
\begin{equation}
i^!:\Sigma\textup{-\underline{Gproj}}\to \A\textup{-\underline{Gproj}}
\end{equation}
whose quasi-inverse is given by the functor $i_\ast: \A\textup{-\underline{Gproj}}\to \Sigma\textup{-\underline{Gproj}}$ which sends every non-projective Gorenstein-projective left $\A$-module $V$ to $\begin{pmatrix} V\\0\end{pmatrix}_0$.
\end{remark}

\section{Proof of the main results}\label{sec3}

Let $\k$ and $\hat{\Ca}$ be as in \S\ref{int}. In order to prove Theorems \ref{thm1} and \ref{thm2}, we use the continuity of the deformation functor (see \cite[Prop. 2.1]{blehervelez}) in order to consider only deformations of modules over Artinian rings in $\hat{\Ca}$.  

\subsection{Proof of Theorem \ref{thm1}}
Assume that $\A$ is a basic connected Nakayama $\k$-algebra without simple projective objects and with $\mathscr{C}(\A)\not=0$, and let $P_{\mathscr{X}(\A)}$, $\A'$ and $H_{\mathscr{X}(\A)}(-)$ be as in (\ref{projective-generator}). Let $R$ be a fixed Artinian ring in $\hat{\Ca}$, and let 
\begin{align*}
P_{R,\mathscr{X}(\A)}=R\otimes_{\k,\iota_R}P_{\mathscr{X}(\A)}&&\text{and}&& H_{R,\mathscr{X}(\A)}(-)=\Hom_{R\A}(P_{R,\mathscr{X}(\A)},-).  
\end{align*}
Note that by Remark \ref{ProjR}, $(P_{R,\mathscr{X}(\A)}, \pi_{R,P_{\mathscr{X}(\A)}})$ is a lift of $P_{\mathscr{X}(\A)}$ over $R$. Moreover, $P_{R,\mathscr{X}(\A)}$ is also a projective $R\A$-module cover of $P_{\mathscr{X}(\A)}$. On the other hand, since $R\A'=\End_{R\A}(P_{R,\mathscr{X}(\A)})^\text{op}$, it follows that the left $R\A$-module $P_{R,\mathscr{X}(\A)}$ is also a right $R\A'$-module. Note also that the functor $H_{R,\mathscr{X}(\A)}(-)$ sends left $R\A$-modules to left $R\A'$-modules. 

\begin{remark}\label{projadd}
Let $P$ be a projective left $\A$-module, and let $P_R$ be as in Remark \ref{ProjR}. If $P$ is an object in $\mathrm{add}\, P_{\mathscr{X}(\A)}$, then $P_R$ is an object of $\mathrm{add}\, P_{R,\mathscr{X}(\A)}$.
\end{remark}

For all left $R\A$-modules $M$, let
\begin{equation}\label{etaM}
\eta_M:P_{R,\mathscr{X}(\A)}\otimes_{R\A'}H_{R,\mathscr{X}(\A)}(M)\to M 
\end{equation}
be defined as $\eta_M(p\otimes \sigma)=\sigma(p)$ for $p\in P_{R,\mathscr{X}(\A)}$ and $\sigma\in H_{R,\mathscr{X}(\A)}(M)$. It is straightforward to prove that $\eta_M$ is a morphism of left $R\A$-modules, which is natural with respect to morphisms $g:M\to M'$ in $R\A$-mod, i.e., 
\begin{equation}\label{naturaleta}
\eta_{M'}\circ (\mathrm{id}_{P_{R,\mathscr{X}(\A)}}\otimes_{R\A'}H_{R,\mathscr{X}(\A)}(g))=g\circ\eta_M.
\end{equation}


Throughout the remainder of this section, we let $V$ be a fixed indecomposable non-projective object in $\mathscr{C}(\A)$. 
\begin{lemma}\label{lemma31}
 Let $(M,\phi)$ be a lift of $V$ over $R$.  
\begin{enumerate}
\item The left $R\A'$-module $(H_{R,\mathscr{X}(\A)}(M), H_{R,\mathscr{X}(\A)}(\phi))$ is a lift of the left $\A'$-module $H_{\mathscr{X}(\A)}(V)$ over $R$, where $H_{R,\mathscr{X}(\A)}(\phi):\k\otimes_RH_{R,\mathscr{X}(\A)}(M)\to H_{\mathscr{X}(\A)}(V)$ is an isomorphism of left $\A'$-modules induced by $\phi$.
\item The morphism $\eta_M$ as in (\ref{etaM}) is an isomorphism of left $R\A$-modules such that $\phi\circ(\mathrm{id}_\k\otimes\eta_M)=\eta_V\circ(\pi_{R,P_{\mathscr{X}(\A)}}\otimes H_{R,\mathscr{X}(\A)}(\phi))$.
\end{enumerate}
\end{lemma}
\begin{proof}
(i). First note that by \cite[Prop. 20.10]{anderson-fuller}, there exists an isomorphism of right $R\A'$-modules
\begin{equation*}
H_{R,\mathscr{X}(\A)}(M)\cong H_{R,\mathscr{X}(\A)}(R\A) \otimes_{R\A} M.
\end{equation*}
Since $H_{R,\mathscr{X}(\A)}(R\A)$ is a projective right $R\A$-module and $M$ is free over $R$, it follows by \cite[Ex. 2.1]{lam} that $H_{R,\mathscr{X}(\A)}(M)$ is projective (and thus free) over $R$. On the other hand, since $H_{R,\mathscr{X}(\A)}(R\A)=R\otimes_\k H_{\mathscr{X}(\A)}(\A)$, we obtain a composition of left $\A'$-module isomorphisms
\begin{align*}
\k\otimes_RH_{R,\mathscr{X}(\A)}(M)&\cong\k\otimes_R((R\otimes_\k H_{\mathscr{X}(\A)}(\A))\otimes_{R\A}M)\cong H_{\mathscr{X}(\A)}(\A)\otimes_\A(\k\otimes_RM)\\
&\cong H_{\mathscr{X}(\A)}(\k\otimes_RM)\xrightarrow{H_{\mathscr{X}(\A)}(\phi)}H_{\mathscr{X}(\A)}(V), 
\end{align*}
which we denote by $H_{R,\mathscr{X}(\A)}(\phi)$. Thus $(H_{R,\mathscr{X}(\A)}(M), H_{R,\mathscr{X}(\A)}(\phi))$ is a lift of the left $\A'$-module $H_{\mathscr{X}(\A)}(V)$ over $R$.

(ii). Let $P^{-1}\xrightarrow{\epsilon^{-1}}P^0\xrightarrow{\epsilon^0} V\to 0$ be a minimal projective presentation of $V$, i.e., $\epsilon^{0}: P^0\to V$ and $\epsilon^{-1}: P^1\to \ker \epsilon^0=\Omega V$ are projective left $\A$-module covers. Note that since $P_{\mathscr{X}(\A)}$ is a projective generator of $\mathscr{C}(\A)$, it follows that $P_1$ and $P_2$ are in $\mathrm{add}\,P_{\mathscr{X}(\A)}$. Using that $(M, \phi)$, $(P^0_R, \pi_{R, P^0})$ and $(P^{-1}_R, \pi_{R, P^{-1}})$ are lifts of $V$, $P^0$ and $P^{-1}$ over $R$, respectively, together with the fact that $\epsilon^0$ and $\epsilon^{-1}$ are essential epimorphisms, we use Remark \ref{ProjR} to obtain a projective presentation $P^{-1}_R\xrightarrow{\alpha^{-1}}P^0_R\xrightarrow{\alpha^0}M\to 0$ of the $R\A$-module $M$ such that $\phi\circ (\mathrm{id}_\k\otimes\alpha^0)=\epsilon^0\circ \pi_{R,P^0}$ and $\pi_{R,P^0}\circ (\mathrm{id}_\k\otimes\alpha^{-1})=\epsilon^{-1}\circ \pi_{R,P^{-1}}$.  Since $P^0_R$ and $P^{-1}_R$ are in $\mathrm{add}\, P_{R,\mathscr{X}(\A)}$ by Remark \ref{projadd}, it follows by \cite[Lemma 29.4]{anderson-fuller} that for $i=-1,0$, the morphism
\begin{equation*}
\eta_{P_R^i}: P_{R,\mathscr{X}(\A)}\otimes_{R\A'}H_{R,\mathscr{X}(\A)}(P_R^i)\to P_R^i
\end{equation*} 
is an isomorphism of left $R\A$-modules, where $\eta_{P_R^i}$ is as in (\ref{etaM}). Since by (\ref{naturaleta}) we have that
\begin{align*}
\eta_M\circ (\mathrm{id}_{P_{R,\mathscr{X}(\A)}}\otimes H_{R,\mathscr{X}(\A)}(\alpha^0))&=\alpha^0\circ \eta_{P^0_R}, \text{  and  } \eta_{P_R^0}\circ (\mathrm{id}_{P_{R,\mathscr{X}(\A)}}\otimes H_{R,\mathscr{X}(\A)}(\alpha^{-1}))=\alpha^{-1}\circ \eta_{P^{-1}_R},
\end{align*}
it follows by the Five Lemma that the morphism $\eta_M$ is an isomorphism of left $R\A$-modules. On the other hand, since 
\begin{equation*}
\k\otimes_R(P_{R,\mathscr{X}(\A)}\otimes_{R\A'}H_{R,\mathscr{X}(\A)}(M))=(\k\otimes_RP_{R,\mathscr{X}(\A)})\otimes_{\A'}(\k\otimes_RH_{R,\mathscr{X}(\A)}(M)),
\end{equation*}
it follows that $(P_{R,\mathscr{X}(\A)}\otimes_{R\A'}H_{R,\mathscr{X}(\A)}(M), \pi_{R,P_{\mathscr{X}(\A)}}\otimes H_{R,\mathscr{X}(\A)}(\phi)))$ is a lift of the left $\A$-module $P_{\mathscr{X}(\A)}\otimes_{\A'}H_{\mathscr{X}(\A)}(V)$ over $R$ which satisfies $\phi\circ(\mathrm{id}_\k\otimes\eta_M)=\eta_V\circ(\pi_{R,P_{\mathscr{X}(\A)}}\otimes H_{R,\mathscr{X}(\A)}(\phi))$. 
\end{proof}

By Lemma \ref{lemma31} we obtain a well-defined map between set of deformations
\begin{equation}\label{naturalR}
\tau_R:\Def_\A(V,R)\to \Def_{\A'}(H_{\mathscr{X}(\A)}(V), R),
\end{equation}
that sends $[M,\phi]\in \Def_\A(V,R)$ to $[H_{R,\mathscr{X}(\A)}(M),H_{R,\mathscr{X}(\A)}(\phi)]$.

\begin{proposition}\label{lemma3.1}
The map $\tau_R$ as in (\ref{naturalR}) is a bijection which is also natural with respect to morphisms $\theta:R\to R'$ between Artinian rings in $\hat{\Ca}$.
\end{proposition}
\begin{proof}

First assume that $(M_1,\phi_1)$ and $(M_2,\phi_2)$ are two lifts of $V$ over $R$ such that there exists a $R\A'$-module isomorphism $g: H_{R,\mathscr{X}(\A)}(M_1)\to H_{R,\mathscr{X}(\A)}(M_2)$ with $H_{\mathscr{X}(\A)}(\phi_1)=H_{\mathscr{X}(\A)}(\phi_2)\circ (\mathrm{id}_\k\otimes g)$. Note that by Lemma \ref{lemma31} (ii), the morphisms $\eta_M$ and $\eta_{M'}$ as in (\ref{projadd}) are isomorphisms of left $R\A$-modules. Let $f:M\to M'$ be the composition of left $R\A$-module isomorphisms $\eta_{M_2}\circ (\mathrm{id}_{P_{R,\mathscr{X}(\A)}}\otimes g)\circ \eta_M^{-1}$. Since for $i=1,2$ we have 
\begin{equation*}
\k\otimes_R(P_{R,\mathscr{X}(\A)}\otimes_{R\A'}H_{R,\mathscr{X}(\A)}(M_i))=(\k\otimes_RP_{R,\mathscr{X}(\A)})\otimes_{\A'}(\k\otimes_RH_{R,\mathscr{X}(\A)}(M_i)),
\end{equation*}
and since 
\begin{equation*}
(\pi_{R,P_{\mathscr{X}(\A)}}\otimes H_{R,\mathscr{X}(\A)}(\phi_2))\circ (\mathrm{id}_{\k\otimes_RP_{R,\mathscr{X}(\A)}}\otimes (\mathrm{id}_\k\otimes g))=\pi_{R,P_{\mathscr{X}(\A)}}\otimes H_{R,\mathscr{X}(\A)}(\phi_1),
\end{equation*}
we obtain together with the second statement of Lemma \ref{lemma31} (ii) that $\phi_2\circ (\mathrm{id}_\k\otimes f)=\phi_1$. This proves that $\tau_R$ is injective. 

Assume next that $(\widetilde{M}, \widetilde{\phi})$ is a lift of $H_{\mathscr{X}(\A)}(V)$ over $R$. Since $P_{R,\mathscr{X}(\A)}$ and $\widetilde{M}$ are both free over $R$, it follows that the left $R\A$-module $P_{R,\mathscr{X}(\A)}\otimes_{R\A'}\widetilde{M}$ is also free over $R$.  Moreover, we have a composition of natural isomorphisms of left $\A$-modules
\begin{align*}
\k\otimes_R(P_{R,\mathscr{X}(\A)}\otimes_{R\A'}\widetilde{M})&=(\k\otimes_RP_{R,\mathscr{X}(\A)})\otimes_{\A'}(\k\otimes_R\widetilde{M})\xrightarrow{\pi_{R,P_{\mathscr{X}(\A)}}\otimes  \widetilde{\phi}}P_{\mathscr{X}(\A)}\otimes_{\A'}H_{\mathscr{X}(\A)}(V)\xrightarrow{\eta_V}V,
\end{align*}
which we denote by $\overline{\phi}$, and thus $(P_{R,\mathscr{X}(\A)}\otimes_{R\A'}\widetilde{M}, \overline{\phi})$ is a lift of $V$ over $R$. By \cite[Prop. 20.10]{anderson-fuller}, we also have natural isomorphisms of left $R\A$-modules 
\begin{align*}
H_{R,\mathscr{X}(\A)}(P_{R,\mathscr{X}(\A)}\otimes_{R\A'}\widetilde{M})&=\Hom_{R\A}(P_{R,\mathscr{X}(\A)},P_{R,\mathscr{X}(\A)}\otimes_{R\A'}\widetilde{M})\\
&=\Hom_{R\A}(P_{R,\mathscr{X}(\A)},P_{R,\mathscr{X}(\A)})\otimes_{R\A'}\widetilde{M}\\
&=R\A'\otimes_{R\A'}\widetilde{M}\\
&=\widetilde{M}, 
\end{align*}
which shows that the lifts $(H_{R,\mathscr{X}(\A)}(P_{R,\mathscr{X}(\A)}\otimes_{R\A'}\widetilde{M}),H_{R,\mathscr{X}(\A)}(\overline{\phi}))$ and $(\widetilde{M}, \widetilde{\phi})$ of $H_{\mathscr{X}(\A)}(V)$ over $R$ are isomorphic. This proves that $\tau_R$ is surjective, and thus $\tau_R$ is a bijection.  

Finally assume that $\theta:R\to R'$ is a morphism between Artinian rings in $\hat{\Ca}$. Then $R'\A=R'\otimes_{R,\theta}R\A$, $P_{R',\mathscr{X}(\A)}=R'\otimes_{R,\theta}P_{R,\mathscr{X}(\A)}$, $H_{\mathscr{X}(\A),R'}(R'\A)=R'\otimes_{R,\theta}H_{R,\mathscr{X}(\A)}(R\A)$, and for all lifts $(M,\phi)$ of $V$ over $R$, we have 
\begin{align*}
H_{\mathscr{X}(\A),R'}(R'\otimes_{R,\theta}M)&=H_{\mathscr{X}(\A),R'}(R'\A)\otimes_{R'\A}(R'\otimes_{R,\theta}M)\\
&=R'\otimes_{R,\theta}(H_{R,\mathscr{X}(\A)}(R\A)\otimes_{R\A}M)=R'\otimes_{R,\theta}H_{R,\mathscr{X}(\A)}(M),
\end{align*}
which implies that $\tau_R$ is natural with respect to morphisms $\theta:R\to R'$ in $\hat{\Ca}$. This finishes the proof of Proposition \ref{lemma3.1}.
\end{proof}
\begin{remark}\label{universal}
The continuity of the deformation functor (see \cite[Prop. 2.1]{blehervelez}) together with Proposition \ref{lemma3.1} imply that the versal deformation rings $R(\A,V)$ and $R(\A', H_{\mathscr{X}(\A)}(V))$ are isomorphic in $\hat{\Ca}$. Moreover, since $R(\A', H_{\mathscr{X}(\A)}(V))$ is universal by \cite[Thm. 1.3]{bleher15}, we also have that $R(\A,V)$ is universal. On the other hand, since $V$ be is an indecomposable non-projective object in $\mathscr{C}(\A)$, it follows that $\Omega V$ is also an indecomposable non-projective object in $\mathscr{C}(\A)$ (see e.g. \cite[Lemma 2.1]{chen-shen-zhou}), and thus we obtain that the versal deformation ring $R(\A,\Omega V)$ of $\Omega V$ is also universal.
\end{remark}  
To finish the proof of Theorem \ref{thm1} (i), we need to prove the following result. 
\begin{proposition}\label{cor3.4}
The universal deformation rings $R(\A,V)$ and $R(\A, \Omega V)$ are isomorphic in $\hat{\Ca}$.
\end{proposition}
\begin{proof}
Let $0\to \Omega V\xrightarrow{\iota} P(V)\xrightarrow{\epsilon} V\to 0$
be a short exact sequence of left $\A$-modules defining $\Omega V$, in particular, $\epsilon$ is an essential epimorphism.  Since $H_{\mathscr{X}(\A)}(-)$ is exact and $H_{\mathscr{X}(\A)}(P(V))$ is a projective left $\A'$-module (see e.g. \cite[Prop. II.2.1 (b)]{auslander}), it follows that $H_{\mathscr{X}(\A)}(\epsilon):H_{\mathscr{X}(\A)}(P(V))\to H_{\mathscr{X}(\A)}(V)$ is an essential epimorphism of $\A'$-modules, which implies that $\Omega_{\A'}H_{\mathscr{X}(\A)}(V)=H_{\mathscr{X}(\A)}(\Omega V)$. On the other hand, since $\A'$ is a basic connected self-injective $\k$-algebra, it follows that $\A'$ is also Frobenius (see e.g. \cite[Cor. 4.3]{skow}). Therefore by Remarks \ref{univ} (iii) and  \ref{universal}, we have that in $\hat{\Ca}$,
\begin{align*}
R(\A,\Omega V)&\cong R(\A', H_{\mathscr{X}(\A)}(\Omega V))\cong R(\A', \Omega_{\A'}H_{\mathscr{X}(\A)}(V))\cong R(\A', H_{\mathscr{X}(\A)}(V))\cong R(\A, V). 
\end{align*}  
\end{proof}

Note that Theorem \ref{thm1} (ii) is an easy consequence of Remark \ref{rem00} and the description of the universal deformation rings of modules over basic self-injective Nakayama algebras given in \cite[Thm. 1.2]{bleher5}. This finishes the proof of Theorem \ref{thm1}. 

In the following example, we verify Theorem \ref{thm1} with a particular basic connected Nakayama $\k$-algebra with no simple projective modules in which there is no overlap (in the sense of \cite{chen-shen-zhou}). 

\begin{example}
Let $\A$ be the basic connected Nakayama $\k$-algebra whose quiver is as in (\ref{cyclequiver}) with $s(\A)=2$ and with Kupisch series $(4, 5)$, i.e., the length of the indecomposable projective left $\A$-modules $P(S_1)$ and $P(S_2)$ are $4$ and $5$, respectively. Note that this algebra coincides the the one studied in \cite[Example 5.10]{chen-shen-zhou}. It follows that $P(S_1)$ is a minimal projective left $\A$-module and there exists an exact sequence of $\A$-modules
\begin{equation*}
0\to E_1\to P(S_1)\to E_1\to 0,
\end{equation*}
where $E_1$ is the indecomposable left $\A$-module with composition series $\begin{matrix}S_1\\S_2\end{matrix}$. Thus $E_1$ is a non-projective (elementary) Gorenstein-projective left $\A$-module. As a matter of fact, $E_1$ is the unique (up to isomorphism) non-projective indecomposable object in $\mathscr{C}(\A)$. Since $\A$ is a monomial algebra in which there is no overlap by \cite[Example 5.10]{chen-shen-zhou}, and $\Omega E_1=E_1$, we obtain by \cite[Thm. 5.2]{bekkert-giraldo-velez} that $R(\A, E_1)=R(\A, \Omega E_1)$ is universal and isomorphic to $\k[\![t]\!]/(t^2)$. On the other hand, it follows by Remark \ref{rem00} that $\A'=\End_\A(P(S_1))^\textup{op}$ is the basic connected self-injective Nakayama algebra with $e(\A')=1$ such that the radical length of the unique indecomposable projective $\A'$-module is given by $\ell\ell(\A')=4/2=2$. Let $S'_1$ denote the unique simple left $\A'$-module. Then  $\Omega_{\A'}S'_1=S'_1$, and $\Ext_{\A'}^1(S'_1, S'_1)=\SHom_{\A'}(\Omega_{\A'}S'_1, S'_1)=\k$. By \cite[Thm. 1.3 (ii)]{bleher15}, we have that the versal deformation ring $R(\A',S_1')=R(\A',\Omega_{\A'}S'_1)$ is universal and isomorphic to $\k[\![t]\!]/(t^2)$. Note also that we have $H_{\mathscr{X}(\A)}(E_1)=S'_1$, and following the notation in Theorem \ref{thm1} (ii), we obtain $\ell_{\mathscr{C}(\A)}=2$ and $d_{\mathscr{C}(\A),E_1}=0$, which gives $n=1$, $m_{E_1}=2$ and $J_1(2)=(t^2)$. This verifies Theorem \ref{thm1} for $\A$. It is also important to mention that $\A$ is a non-self-injective Gorenstein $\k$-algebra with injective dimension $2$.
\end{example}

\subsection{Proof of Theorem \ref{thm2}}

Let $\k$ and $\hat{\Ca}$ be as in \S\ref{int}.  We first need the following result.

\begin{lemma}\label{lemma3.7}
Let $\Sigma$ be a triangular matrix $\k$-algebra as in (\ref{triangalgebra}).
\begin{enumerate}
\item For all morphisms $\theta: R\to R'$ in $\hat{\Ca}$, and all left $R\Sigma$-modules $\begin{pmatrix} M\\N\end{pmatrix}_g$, there exists and isomorphism of left $R'\Sigma$-modules 
\begin{equation*}
R'\otimes_{R,\theta}\begin{pmatrix} M\\N\end{pmatrix}_g\cong \begin{pmatrix} R'\otimes_{R,\theta}M\\R'\otimes_{R,\theta}N\end{pmatrix}_{\mathrm{id}_{R'}\otimes g}
\end{equation*}
\item Let $\begin{pmatrix} V\\W\end{pmatrix}_f$ be a left $\Sigma$-module and let $R$ be a ring in $\hat{\Ca}$. Then $\left(\begin{pmatrix} M\\N\end{pmatrix}_g, \begin{pmatrix} \phi_M\\\phi_N\end{pmatrix}\right)$ is a lift of $\begin{pmatrix} V\\W\end{pmatrix}_f$ over $R$ if and only if $(M, \phi_M)$ and $(N, \phi_N)$ are lifts of $V$ and $W$ over $R$, respectively.  
\end{enumerate}
\end{lemma}

\begin{proof}
(i). Let $\theta:R'\to R$ be a morphism in $\hat{\Ca}$ and let $\begin{pmatrix} M\\N\end{pmatrix}_g$ be a left $R\Sigma$-module. Then the map
\begin{equation*}
\Theta: R'\otimes_{R,\theta}\begin{pmatrix} M\\N\end{pmatrix}_g\to \begin{pmatrix} R'\otimes_{R,\theta}M\\R'\otimes_{R,\theta}N\end{pmatrix}_{\mathrm{id}_{R'}\otimes g}
\end{equation*}
defined for all $r'\in R'$, $x\in M$, $y\in N$ as  
\begin{equation*}
\Theta\left(r'\otimes\begin{pmatrix}x\\y\end{pmatrix}\right)=\begin{pmatrix}r'\otimes x\\r'\otimes y\end{pmatrix}
\end{equation*}
is a isomorphism of $R'\Sigma$-modules.
 
(ii). Assume first that $\left(\begin{pmatrix} M\\N\end{pmatrix}_g, \begin{pmatrix} \phi_M\\\phi_N\end{pmatrix}\right)$ is a lift of $\begin{pmatrix} V\\W\end{pmatrix}_f$ over $R$. By looking at the exact sequence of left $R\Sigma$-modules
\begin{equation}\label{freelift}
0\to \begin{pmatrix} M\\0\end{pmatrix}_0\to \begin{pmatrix}M\\N\end{pmatrix}_g\to \begin{pmatrix}0\\N\end{pmatrix}_0\to 0
\end{equation}
and by using Nakayama's Lemma, we obtain that since $\begin{pmatrix}M\\N\end{pmatrix}_g$ is free over $R$, then so are $\begin{pmatrix} M\\0\end{pmatrix}_0$ and $\begin{pmatrix} 0\\N\end{pmatrix}_0$. This clearly implies that $M$ and $N$ are both free over $R$. By using (i), we can identify $\k\otimes_R\begin{pmatrix}M\\N\end{pmatrix}_g$ with $\begin{pmatrix}\k\otimes_RM\\\k\otimes_RN\end{pmatrix}_{\mathrm{id}_\k\otimes g}$. Since $\begin{pmatrix} \phi_M\\\phi_N\end{pmatrix}: \begin{pmatrix}\k\otimes_RM\\\k\otimes_RN\end{pmatrix}_{\mathrm{id}_\k\otimes g}\to \begin{pmatrix} V\\W\end{pmatrix}_f$ is an isomorphism of left $\Sigma$-modules, it follows that $\phi_M:\k\otimes_RM\to V$ and $\phi_N:\k\otimes_RN\to N$ are isomorphisms of left $\A$-modules and left $\G$-modules, respectively. Thus $(M, \phi_M)$ and $(N,\phi_N)$ are lifts of $V$ and $W$ over $R$, respectively. Note that this argument is symmetric. This finishes the proof of Lemma \ref{lemma3.7}.

\end{proof}

In the following, we assume that $B$ is projective as a left $\A$-module and $\Sigma$ is Gorenstein with $\G$ of finite global dimension. Let $\begin{pmatrix} V\\W\end{pmatrix}_f$ be a non-projective Gorenstein-projective left $\Sigma$-module. By Theorem \ref{zhangtheorem}, we have a short exact sequence of Gorenstein-projective left $\A$-modules 
\begin{equation*}
0\to B\otimes_\G W\xrightarrow{f} V\to \mathrm{coker}\,f\to 0.
\end{equation*}
Let $R$ be a fixed Artinian ring in $\hat{\Ca}$ and let $\left(\begin{pmatrix} M\\ N\end{pmatrix}_g, \begin{pmatrix} \phi_M\\\phi_N\end{pmatrix}\right)$ be a lift of $\begin{pmatrix}V\\W\end{pmatrix}_f$ over $R$. Thus $(M,\phi_M)$ and $(N, \phi_N)$ are lifts of $V$ and $W$ over $R$, respectively.  Since $\G$ is of finite global dimension, it follows that $W$ is a projective left $\G$-module. Let $W_R=R\otimes_\k W$ and $\pi_{W,R}: \k\otimes_R W_R\to W$ be the natural isomorphism. Then $(W_R, \pi_{R,W})$ is a lift of $W$ over $R$. Since we also have that $R(\G, W)=\k$ by Remark \ref{univ} (i), it follows that $(N,\phi_N)$ and $(W_R, \pi_{R,W})$ are isomorphic lifts of $W$ over $R$. Thus without loss of generality, we can assume that $N=W_R$ and $\phi_N=\pi_{R,W}$. Note that $W_R$ is a projective left $R\G$-module, and this implies that $B_R\otimes_{R\G} W_R$ is a projective left $R\A$-module, for $B_R$ is also a projective left $R\A$-module. We define a map between sets of deformations   
\begin{equation}
m_{i^!,R}:\mathrm{Def}_\Sigma\left(\begin{pmatrix} V\\W\end{pmatrix}_f,R\right)\to \mathrm{Def}_\A(V,R)
\end{equation}
induced by the operator $i^!: \Sigma\textup{-Gproj}\to\A\textup{-Gproj}$ as in (\ref{functorgor}) as follows. For all lifts $\left(\begin{pmatrix} M\\ W_R\end{pmatrix}_g, \begin{pmatrix} \phi_M\\\pi_{W,R}\end{pmatrix}\right)$ of $\begin{pmatrix} V\\W\end{pmatrix}_f$ over $R$, define $m_{i^!,R}\left(\left[\begin{pmatrix} M\\ W_R\end{pmatrix}_g, \begin{pmatrix} \phi_M\\\pi_{R,W}\end{pmatrix}\right]\right)=[M,\phi_M]$.
By using Lemma \ref{lemma3.7}, it is straightforward to show that $m_{i^!,R}$ is a bijection that is natural with respect to morphisms $\alpha:R\to R'$ of Artinian objects in $\hat{\Ca}$. The continuity of the deformation functor implies that the versal deformation rings $R\left(\Sigma, \begin{pmatrix} V\\W\end{pmatrix}_f\right)$ and $R(\A,V)$ are isomorphic in $\hat{\Ca}$. This finishes the proof of Theorem \ref{thm2}.

The following example applies Theorem \ref{thm2} to a particular triangular matrix algebra. 

\begin{example}
Let $\Sigma$ be the $\k$-algebra defined by the quiver
\begin{equation*}
Q: \xymatrix{\underset{1}{\bullet}\ar[r]^{\alpha}&\underset{2}{\bullet}\ar[r]^{\beta}&\underset{3}{\bullet}\ar@(ur,dr)^{\gamma}} 
\end{equation*}
with the relation $\rho=\{\gamma^3\}$. Then it follows from \cite[Example 3.6]{xiong-zhang} that $\Sigma$ is a triangular matrix $\k$-algebra as in (\ref{triangalgebra}) with $\A=\k[x]/(x^3)$, $\G$ is given by the quiver $\underset{1}{\bullet}\xrightarrow{\alpha}\underset{2}{\bullet}$ and $B=e_3\Sigma (1-e_3)$. Moreover, $B\cong \A\oplus \A$ as left $\A$-modules, and $B\cong e_2\G\oplus e_2\G$ as right $\G$-modules, and the indecomposable non-projective Gorenstein-projective left $\Sigma$-modules are given by the following representations:
\begin{align*}
U_1=\xymatrix@1@=20pt{0\ar[r]&0\ar[r]&\k\ar@(ur,dr)^{0}}, &&U_2=\xymatrix@1@=20pt{0\ar[r]&0\ar[r]&\k^2\ar@(ur,dr)^{\tiny\begin{pmatrix}0&0\\1&0\end{pmatrix}}}.
\end{align*}
Note that $\Omega_\Sigma U_1=U_2$, and that $\End_\Sigma(U_1)=\k=\SEnd_\Sigma(U_2)$. Moreover, for $i=1,2$, we have that $\Ext_\Sigma^1(U_i,U_i)\cong \SHom_\Sigma(\Omega_\Sigma U_i,U_i)\cong \k$. All this together with Remark \ref{univ} (i), (ii) imply that for $i=1,2$, the versal deformation ring $R(\Sigma, U_i)$ is universal and a quotient of $\k[\![t]\!]$.  
Note that $\A$ is a basic connected self-injective Nakayama $\k$-algebra whose quiver with relations is given by $\left(\xymatrix{\underset{3}{\bullet}\ar@(ur,dr)^{\gamma}}, \{\gamma^3\}\right)$. Since $\G$ has finite global dimension and $B$ is projective as a left $\A$-module and as a right $\G$-module, it follows that $\Sigma$ is also a non-self-injective Gorenstein $\k$-algebra of infinite global dimension. Thus we can find the isomorphism classes of $R(\Sigma, U_1)$ and $R(\Sigma, U_2)$ by using Theorem \ref{thm2} as follows. There are only two isomorphism classes of indecomposable non-projective left $\A$-modules, namely, the simple left $\A$-module $V_1=S_3$, and the left $\A$-module $V_2$ whose composition series is $\begin{matrix} S_3\\S_3\end{matrix}$. Note that $\Omega V_1=V_2$. By \cite[Thm. 1.3 (ii)]{bleher15} or Theorem \ref{thm1}, it follows that the versal deformation rings $R(\A,V_1)$ and $R(\A,V_2)$ are both isomorphic to $\k[\![t]\!]/(t^3)$, which together with Theorem \ref{thm2}, imply that the universal deformation rings $R(\Sigma, U_1)$ and $R(\Sigma, U_2)$ are both isomorphic to $\k[\![t]\!]/(t^3)$.
\end{example}
\subsection*{Acknowledgments}
The author would like to express his gratitude to C. M. Ringel for providing a clarification concerning the functor $H_{\mathscr{X}(\A)}(-)$ as in (\ref{projective-generator}), which was used to prove Theorem \ref{thm1}, and to F. M. Bleher for valuable discussions concerning this note.   

\bibliographystyle{amsplain}
\bibliography{DeformationsNakayamaMatrix2}

\end{document}